\documentclass{amsart}
\usepackage[T1]{fontenc}
\usepackage{textcomp}
\usepackage[utf8x]{inputenc}
\usepackage[english]{babel}
\usepackage{ucs}
\usepackage{graphicx}
\usepackage{mathrsfs}
\usepackage{rotating} 
\usepackage{tikz}
\usepackage{amssymb}
\usepackage[titletoc,toc]{appendix}
\usepackage{subcaption}
\numberwithin{equation}{section} 
\newtheorem{thm}{Theorem}[section]

\newtheorem{cor}[thm]{Corollary}
\newtheorem{lem}[thm]{Lemma}
\newtheorem{prop}[thm]{Proposition}
\newtheorem{quest}[thm]{Question}

\theoremstyle{definition}
\newtheorem{definition}[thm]{Definition}
\theoremstyle{remark}
\newtheorem{rem}[thm]{Remark}
\numberwithin{equation}{section}

\newcommand{\PP}{\mathbb{P}}
\newcommand{\FF}{\mathbb{F}}
\newcommand{\FFbar}{\overline{\mathbb{F}}}

\newcommand{\Qell}{\mathbb{Q}_{\ell}}
\newcommand{\Ccal}{\mathcal{C}}

\newcommand{\Hetc}[2]{H^{#1}_{\text{\'et},c}(#2,\Qell)}
\newcommand{\Frob}{F}
\newcommand{\Pic}{\mathrm{Pic}}

\makeindex
\begin{document}

\title{Seven points in general linear position}%
\author{Olof Bergvall}%
\address{Department of Electronics, Mathematics and Natural Sciences, University of Gävle, Kungsbäcksvägen 47,
80176 Gävle, Sweden}
\email{olof.bergvall@hig.se}

\begin{abstract}
We determine the cohomology groups of the space of seven points in general linear position
as representations of the symmetric group on seven elements by making equivariant point
counts over finite fields. We also comment on the case of eight points.
\end{abstract}

\maketitle

\section{Introduction}
Given a variety $X$, it is very natural to consider $m$-tuples $(P_1,\ldots, P_m)$ of points
on $X$. If no condition is placed on the tuples, the space of these tuples is $X^m$
and the space of tuples such that the points are distinct is the configuration space $C_m(X)$.
If the order of the points is irrelevant one instead arrives at the symmetric product $S^m(X)$
and the unordered configuration space $B_m(X)$, respectively.
The ordered and unordered spaces are naturally related through the action of the symmetric group
$S_m$ permuting the points. 

It is both possible and interesting to pose more refined requirements on the $m$-tuples.
For instance, tuples of points in general position, i.e. such that there
is no ``unexpected'' subvariety containing the points, have been studied
extensively in the classical literature. See the book \cite{dolgachevortland}
of Dolgachev and Ortland for modern account of this topic and for further references.
Tuples of points in general position also have close connections to
various moduli spaces of curves, surfaces and abelian varieties, see e.g.
\cite{bergstrombergvall}, \cite{bergvall_gd} and \cite{looijenga}.
We also mention that spaces of $m$-tuples in general position in $\PP^2$ have
been studied from a cohomological point of view by Gounelas
and the author when $m$ is at most $7$,
see \cite{bergvall_gd}, \cite{bergvall_pts} and \cite{bergvallgounelas}. 
The techniques employ both point counts over finite fields and purity arguments,
two topics which will be discussed further in the present paper.

In $\PP^n$ it is also natural to consider $m$-tuples of points in general
\emph{linear} position. In other words, we require that no subset of $n+1$ points
should lie on a hyperplane. We denote the space of $m$-tuples of points
in $\PP^n$ which are in general linear position by $C_{n,m}$ and its quotient
by $\mathrm{PGL}(n+1)$ by $\Ccal_{n,m}$.  \footnote{To avoid technical subtleties we assume
$m \geq n+2$.}
In fact, $C_{m,n} \cong \mathrm{PGL}(n+1) \times \Ccal_{n,m}$
so for many purposes it is merely a matter of convenience (or preference) which one to use.
In this work, the spaces $\Ccal_{n,m}$ are more convenient since they (sometimes) satisfy
certain cohomological purity properties which the spaces $C_{m,n}$ do not satisfy.
Also the spaces $\Ccal_{n,m}$ have received considerable attention classically, see e.g. the work of
Dolgachev and Ortland \cite{dolgachevortland}.
From a more modern perspective, the spaces $\Ccal_{n,m}$ can be viewed as
natural generalizations of the moduli spaces $\mathcal{M}_{0,n}$ of
$n$-pointed rational curves. In fact, sending a $(n+3)$-tuple of points in $\PP^n$
to the isomorphism class of the rational normal curve passing through them
gives an isomorphism $\Ccal_{n,n+3} \cong \mathcal{M}_{0,n+3}$. The
spaces $\Ccal_{n,m}$, especially the case $n=2$, have also been studied 
extensively from the perspective of coding theory starting with the work of
Glynn \cite{glynn} and continued in the work of Iampolskaia, Skorobogatov and Sorokin
\cite{iampolskaiaetal} and Kaplan et al \cite{kaplanetal}. In particular,
the number $|\Ccal_{2,m}(\FF_q)|$ of $m$-tuples of points in general linear position
in the projective plane over the finite field $\FF_q$ with $q$ elements is known
for all $q$ and $m<10$ through these works. 

Through Lefschetz trace formula, these point counts determine
the Euler characterics of the spaces $\Ccal_{n,m}$. By also
considering the action of the group $S_m$ these Euler characteristics can
be decomposed into virtual representations of $S_m$. From an arithmetic
point of view, this allows one to obtain more refined enumerative information
about the counted objects. For an example in this direction, see the recent paper of Das
\cite{das}. From a more topological or representation theoretic perspective,
the $S_m$-equivariant information encodes cohomological information about
quotients of $\Ccal_{n,m}$ by subgroups of $S_m$ - in particular, one
obtains the Euler characteristic of the ``unordered space'' $\Ccal_{n,m}/S_m$.

If $m \leq 6$ it has been shown, through work of 
Gounelas and the author \cite{bergvallgounelas} and Das and O'Connor  \cite{dasoconnor}
that $\Ccal_{2,m}$ satisfies a strong condition called minimal purity.
This purity condition implies that the $S_m$-equivariant point counts
of $\Ccal_{2,m}$ do not only determine the Euler characteristic as
a virtual representation of $S_m$ but actually determine each cohomology
group $H^i(\Ccal_{2,m})$ as a representation of $S_m$.

The purpose of the present paper is to extend these results to the case
of $7$ points. Our main result is the following.

\begin{thm}
\label{27thm}
 The cohomology groups $H^k(\Ccal_{2,7})$ of the space $\Ccal_{2,7}$ 
 of seven points in general linear position are pure of Hodge-Tate type
 $(k,k)$. The cohomology groups are described explicitly in
 Table~\ref{cohtab} as representations of the symmetric group $S_7$.
\end{thm}

\begin{rem}
 In most of the paper, we work over a field of positive (odd) characteristic
 and use compactly supported étale cohomology. We have chosen to formulate the
 result in characteristic 0 and using ordinary (de Rham) cohomology since it
 may be less intimidating to some readers. One reaches the present form of the result
 via a standard argument using Artin's comparison theorem (as well as constructibility and base change)
 and Poincaré duality. The precise formulations in positive characteristic is given later
 in the paper.
\end{rem}

We prove Theorem~\ref{27thm} essentially by first relating the space $\Ccal_{2,7}$ to
the complement of a certain toric arrangement (the cohomology of such complements
have been studied extensively) and then counting points of $\Ccal_{2,7}$ over finite fields
and applying Lefschetz trace formula.

We also discuss the case of $8$ points. In this case, we are able to prove the following.

\begin{thm}
 The cohomology groups $H^k(\Ccal_{2,8})$ of the space $\Ccal_{2,8}$ 
 of eight points in general position are pure of Hodge-Tate type
 $(k,k)$.
\end{thm}

Our method
fails to prove the analogous result for the space of $8$ points in general linear position.
This suggests that the transition in mixed Hodge structure from minimal purity to more general 
behaviour occurs between $7$ and $8$ points for points in general linear position
while the transition for points in general position seems to be between $8$ and $9$.
To decide whether this is actually the case would of course be very interesting.
More generally, we have the following question.

\begin{quest}
\label{question}
 For which integers $m$ and $n$ and for which generality conditions $gc$
 is the space $\Ccal_{n,m}^{gc}$ of $m$-tuples of points in $\PP^n$
 satisfying $gc$ minimally pure? 
\end{quest}

We also remark that Propositions~\ref{27minpure_prop} and \ref{28minpure_prop}, especially the proofs,
are closely related to moduli of surfaces with anticanonical cycles.
Such moduli spaces have been studied quite extensively, e.g. in the work \cite{looijenga81}
and, more recently, by Gross, Hacking and Keel \cite{grosshackingkeel}. The present work
sheds new light on the cohomology of some of these spaces. We also expect the present paper
to be of use in the pursuit of counting pairs of curves with prescribed intersection,
see e.g. the recent paper of Kaplan and Matei, \cite{kaplanmatei}.


\section{Background}
Let $p$ be an odd prime number, let $n\geq 1$ be an integer and let $q=p^n$.
Let $\FF_q$ be a finite field with $q$ elements, let
$\FF_{q^m}$ be a degree $m$ extension of $\FF_q$ and
let $\FFbar_{q}$ be an algebraic closure of $\FF_{q}$.
Let $X$ be a scheme defined over $\FF_{q}$, let $\overline{X}=X\otimes \FFbar_q$ and 
let $\Frob$ denote its geometric Frobenius endomorphism.
Finally, let $\ell$ be prime number different from $p$ and
let $\Hetc{k}{X}$ denote the $k$-th compactly supported étale cohomology
group of $X$ with coefficients in $\Qell$.

If $\Gamma$ is a finite group acting on $X$ by $\FF_q$-rational automorphisms, then each cohomology group 
$\Hetc{k}{X}$ is a $\Gamma$-representation. Lefschetz trace formula allows us to obtain information about these representations by counting the number of fixed points of 
$\Frob \sigma$ for different $\sigma \in \Gamma$.
 
\begin{thm}[Lefschetz trace formula]
\label{lefschetz}
 Let $X$ be a separated scheme of finite type over 
 $\FF_q$ with Frobenius endomorphism $\Frob$ and let $\sigma$
 be a rational automorphism of $X$ of finite order. Then
 \begin{equation*}
  |\overline{X}^{\Frob \sigma}| = 
  \sum_{k \geq 0} (-1)^k \cdot \mathrm{Tr}\left(\Frob \sigma, \Hetc{k}{\overline{X}} \right),
 \end{equation*}
 where $\overline{X}^{\Frob \sigma}$ denotes the fixed point set of $\Frob \sigma$.
\end{thm}

 Under the assumptions of the theorem, the number $|\overline{X}^{\Frob \sigma}|$ only depends on the conjugacy class of $\sigma$.

 Let $R(\Gamma)$ denote the representation ring of $\Gamma$ and
 let the compactly supported $\Gamma$-equivariant Euler characteristic of $X$ be 
 defined as the virtual representation
 \begin{equation*}
  \mathrm{Eul}^{\Gamma}_{\overline{X},c} = \sum_{k \geq 0} (-1)^k \cdot \Hetc{k}{\overline{X}} \in R(\Gamma).
 \end{equation*}
By evaluating $\mathrm{Eul}^{\Gamma}_{\overline{X},c}$ at an element $\sigma \in \Gamma$
we mean
 \begin{equation*}
 \mathrm{Eul}^{\Gamma}_{\overline{X},c}(\sigma) = 
  \sum_{k \geq 0} (-1)^k \cdot \mathrm{Tr}\left( \sigma, \Hetc{k}{\overline{X}} \right) \in \mathbb{Z}
 \end{equation*}
 Note in particular that $ \mathrm{Eul}^{\Gamma}_{\overline{X},c}(\mathrm{id})$ is
 the ordinary Euler characteristic of $\overline{X}$.
 By character theory, $\mathrm{Eul}^{\Gamma}_{\overline{X},c}$ is completely
 determined by computing $ \mathrm{Eul}^{\Gamma}_{\overline{X},c}(\sigma)$
 for a representative $\sigma$ of each conjugacy class of $\Gamma$. 
 
 Thus, phrased in these terms Lefschetz trace formula says that the $\Gamma$-equivariant Euler characteristic
 of $X$ is determined by counting fixed points of $\Frob \sigma$ for a representative $\sigma$
 of each conjugacy class of $\Gamma$.
 If we make stronger assumptions on $X$ we can say more about its cohomology from these equivariant
 point counts.
 
 \begin{definition}[Dimca and Lehrer \cite{dimcalehrer}]
 \label{poscharmp}
  Let $X$ be an irreducible and separated scheme of finite type over $\FFbar_{q}$. 
  The scheme $X$ is called
  \emph{minimally pure} if $\Frob$ acts on $\Hetc{k}{X}$ by multiplication by 
  $q^{k-\mathrm{dim}(X)}$.
  
  A pure dimensional and separated scheme $X$ of finite type over $\FFbar_{q}$ is minimally
  pure if for any collection $\{X_1, \ldots, X_r\}$ of irreducible components of 
  $X$, the irreducible scheme $X_1 \setminus (X_2 \cup \cdots \cup X_r)$ is
  minimally pure.
 \end{definition}
 
 We define the compactly supported $\Gamma$-equivariant Poincaré polynomial of $X$
 as
 \begin{equation*}
  P^{\Gamma}_{\overline{X},c}(t) = \sum_{k \geq 0} \Hetc{k}{\overline{X}} \cdot t^k \in R(\Gamma)[t].
 \end{equation*}
and we introduce the notation
 \begin{equation*}
 P^{\Gamma}_{\overline{X},c}(\sigma)(t) =
  \sum_{k \geq 0} \mathrm{Tr}\left( \sigma, \Hetc{k}{\overline{X}} \right) \cdot t^k \in \mathbb{Z}[t]
 \end{equation*}
 
 Thus, if $X$ is minimally pure 
 then the eigenvalues of $\Frob \sigma$ on $\Hetc{k}{\overline{X}}$ are of the form 
 $\zeta q^{k-\mathrm{dim}(X)}$ for some root of unity $\zeta$. Thus, a term $q^{k-\mathrm{dim}(X)}$ in $|\overline{X}^{\Frob \circ \sigma}|$
 can only come from $\Hetc{k}{\overline{X}}$ and we can determine the $\Gamma$-equivariant Poincaré
 polynomial of $X$ via the relation
 \begin{equation*}
  \mathrm{Eul}^{\Gamma}_{\overline{X},c}(\sigma) = q^{-2\mathrm{dim}(X)} \cdot P^{\Gamma}_{\overline{X},c}(\sigma)(-q^2).
 \end{equation*}
 In particular we see that if $X$ is minimally pure, then we can determine the cohomology groups
 of $X$ by counting points over finite fields. 
 
 Since $q$ is an integer we see that 
 the above formula indeed yields an integer. On the other hand we may also
 view the above expression as a polynomial in $q$. We then see that the point
 count is polynomial and that the coefficients of this polynomial is given by the
 values of the characters of the $\Gamma$-representations $\Hetc{k}{\overline{X}}$.

\section{The space of seven points}
\label{sevepointsec}
Let $T=(P_1, \ldots, P_7)$ be a septuple of points in $\PP^2$. 
We say that $T$ is in general linear position if no three of
the points lie on a line. We denote the space of septuples
in $\PP^2$ in general linear position, up to projective equivalence,
by $\Ccal_{2,7}$. 

\begin{prop}
\label{27minpure_prop}
 The space $\Ccal_{2,7}$ of seven points in the plane in general linear position is minimally pure.
\end{prop}

\begin{proof}
 The idea of the proof is as follows. We shall construct a $\FF_q$-rational and $S_7$-equivariant finite cover
 $\varphi: X\to \Ccal_{2,7}$ such that $X$ is minimally pure.
 A finite cover $\varphi: X \to \Ccal_{2,7}$ induces an
 $\FF_q$-rational inclusion $\Hetc{k}{\Ccal_{2,7}} \hookrightarrow \Hetc{k}{X}$.
 Thus, the image of $\Hetc{k}{\Ccal_{2,7}}$ is preserved by the Frobenius.
 Since $X$ is minimally pure, $\Frob$ acts on $\Hetc{k}{X}$ with all eigenvalues equal to
 $q^{k-\mathrm{dim}(X)}$. We then see that $\Frob$ also acts with all eigenvalues
 equal to $q^{k-\mathrm{dim}(X)}$ on the subspace $\Hetc{k}{\Ccal_{2,7}}$, i.e. 
 that $\Ccal_{2,7}$ is minimally pure.
 
 To construct the space $X$ we observe that if we blow up $\PP^2$ in seven
 points $P_1, \ldots, P_7$ in general linear position we obtain a weak Del Pezzo
 surface $S$ of degree $2$ marked with seven $(-1)$-curves $E_1, \ldots, E_7$. 
 Together with $L$, the strict transform of a line, these curves constitute a
 basis $B$ for Picard group of $S$. Such a basis for $\Pic(S)$ (i.e. one arising in this way
 from a blow up) is called a geometric
 marking. Let $\pi:S \to \PP^2$ denote the blow down morphism.
 We rigidify the situation further by marking $S$ with an anticanonical curve $C$,
 i.e. a curve whose class is $3L-E_1-\cdots-E_7$. 
 The curve $A=\pi(C)$ has degree $3$. Thus, if $A$ is reducible it will
 consist of two components $A_1$ and $A_2$ of degrees $1$ and $2$, respectively. 
 
 Let $X$ be the space of triples $(S,B,C)$ such that $S$ is a weak Del Pezzo surface of degree $2$,
 $B$ is a geometric marking of $S$ and $C$ is an anticanonical curve on $S$ such
 that $A=\pi(C)$ is reducible and the degree $1$ component $A_1$ of $A$ contains $2$
 of the points $P_1, \ldots, P_7$.
 The morphism $\varphi:X \to \Ccal_{2,7}$ sending a triple
 $(S,B,C)$ to the septuple $(P_1,\ldots, P_7)$ defined by the blow down of $B$
 is clearly $\FF_q$-rational, finite (of degree $21$) and $S_7$-equivariant. It thus remains to
 show that $X$ is minimally pure.
 
 Note that the curve $A_1$ either will intersect the degree $2$ component $A_2$ transversally in two points
 or $A_1$ will be tangent to $A_2$ at some point. The points of intersection may or may not coincide with some
 of the points $P_1, \ldots, P_7$, see Figure~\ref{Acurve}. Let
 $X^{\textrm{gen}}$ denote the subspace of $X$ consisting of triples
 $(S,B,C)$ such that $A_1$ and $A_2$ intersect transversally and let
 $X^{\textrm{tan}}$ be the subspace of $X$ consisting of triples such that
 $A_1$ is a tangent to $A_2$. We then have that $X$ is the disjoint union
 $$
  X = X^{\textrm{gen}} \sqcup X^{\textrm{tan}}.
 $$
 
 The next step is to show that $X^{\textrm{gen}}$ is isomorphic to a disjoint union complements of
 toric arrangements in algebraic tori $T_i$ and that $X^{\textrm{tan}}$ is isomorphic to a disjoint union of
  projectivizations of complements of hyperplane arrangements in vector spaces $V_i$
 in such at way that $V_i$ can be naturally be identified with the tangent space of $T_i$ at the identity.
 This follows exactly as in Section 1 of \cite{looijenga} 
 (see in particular Proposition 1.13, Proposition 1.15 and Proposition 1.17) with the
 very small modification that we need to take away the conditions corresponding to disallowing more
 than $5$ of the points $P_1, \ldots, P_7$ to lie on a conic.
 See also \cite{bergvallgounelas}, Theorem 5.1, for a slightly different account of the same ideas.
 
 It now directly follows from Lemma 3.6 of \cite{looijenga} that there is an inclusion
 $$
 \Hetc{k}{X} \hookrightarrow  \Hetc{k}{X^{\textrm{gen}}}
 $$
 Complements of toric arrangements are well known to be minimally pure (see e.g. \cite{dimcalehrer})
 so it follows that also $X$ is minimally pure. This completes the proof.
\end{proof}

\begin{figure}[h!]
  \centering
  \begin{subfigure}[b]{0.3\linewidth}
    \includegraphics[width=\linewidth]{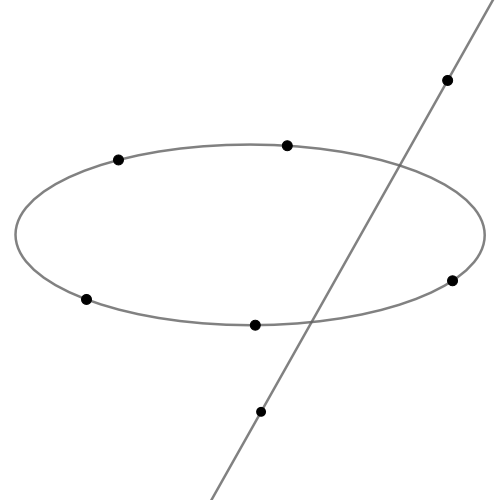}
    \caption{}
  \end{subfigure}
  \begin{subfigure}[b]{0.3\linewidth}
    \includegraphics[width=\linewidth]{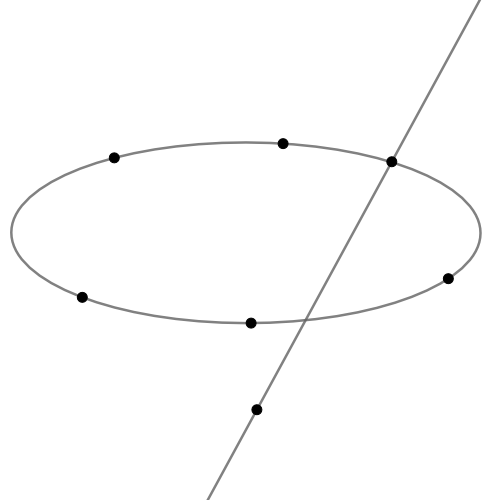}
    \caption{}
  \end{subfigure}
  \begin{subfigure}[b]{0.3\linewidth}
    \includegraphics[width=\linewidth]{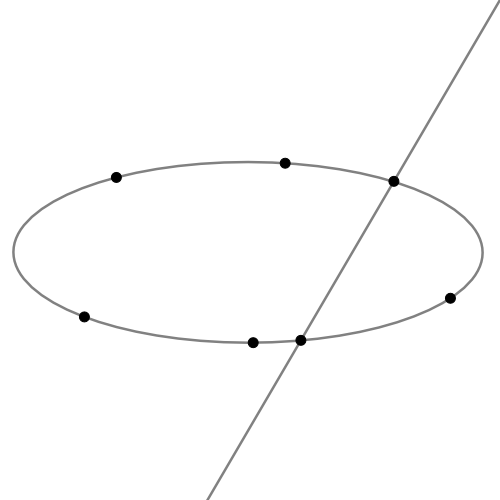}
    \caption{}
  \end{subfigure}
  \begin{subfigure}[b]{0.3\linewidth}
    \includegraphics[width=\linewidth]{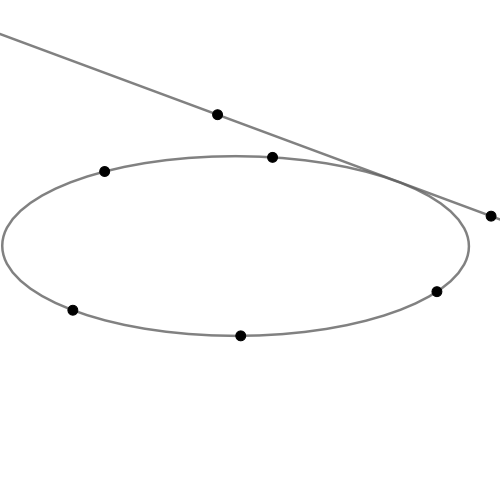}
    \caption{}
  \end{subfigure}
  \begin{subfigure}[b]{0.3\linewidth}
    \includegraphics[width=\linewidth]{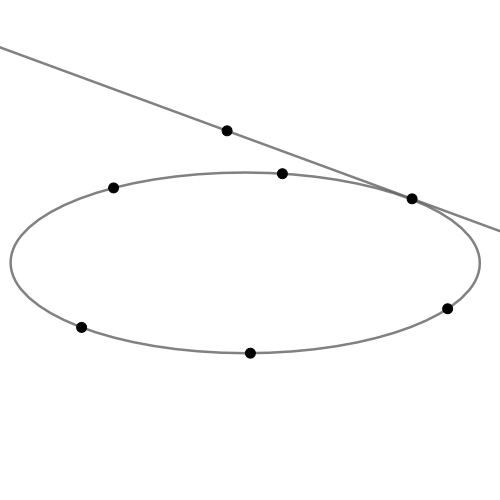}
    \caption{}
  \end{subfigure}
  \caption{The curve $A=A_1 \cup A_2$.}
  \label{Acurve}
\end{figure}

By Proposition~\ref{27minpure_prop}, we can determine the cohomology groups of
$\Ccal_{2,7}$ as representations of $S_7$ by counting fixed points of $\Frob \sigma$ 
for each $\sigma\in S_7$ and then applying Lefschetz trace formula.
Since the number $|\Ccal_{2,7}^{\Frob \sigma}|$ only depends on the
conjugacy class of $\sigma$ we only have to do one computation for each of the $15$
conjugacy classes of $S_7$. We denote the conjugacy class of $\sigma$ by $[\sigma]$.
Recall also that the conjugacy class $[\sigma]$ is determined by the cycle type
of $\sigma$. We denote cycle types by $7^{n_7}6^{n_6} \cdots 2^{n_2}1^{n_1}$
(with the convention to omit $m^{n_m}$ if $n_m=0$ and to write $m$ instead of
$m^{n_m}$ if $n_m=1$).

The necessary computations can be deduced quite straightforwardly from the corresponding
computations in \cite{bergvall_pts}. We therefore omit the details in most cases and
simply record the results in Table~\ref{counttab}. However, to illustrate the nature of the
computations, at least in one of the simpler cases, we include the following lemma.

\begin{lem}
 If $[\sigma]=[7]$, then $|\Ccal_{2,7}^{\Frob \sigma}| = q^6 + q^4 + q^3 + q^2 + 1$.
\end{lem}

\begin{proof}
 Since the result is independent of the particular representative of the conjugacy class,
 we pick $\sigma$ so that $\sigma^{-1}=(1234567)$.
 Let $X=(\PP^2)^7$, let $\Delta$ denote the subspace of $X$ consisting of septuples 
 $(P_1,\ldots, P_7)$ such that at least $3$ of the points lie on a line and let
 $U=X \setminus \Delta$. Then $\Ccal_{2,7}^{\Frob \sigma} =U^{\Frob \sigma}/\mathrm{PGL}(3,\FF_q)$.
 Since $|U^{\Frob \sigma}|=|X^{\Frob \sigma}|-|\Delta^{\Frob \sigma}|$ we can obtain the result
 by counting $X^{\Frob \sigma}$ and $\Delta^{\Frob \sigma}$.
 
 A septuple $(P_1, \ldots, P_7)$ is fixed by $\Frob \sigma$ if and only if
 $\Frob$ acts as $\sigma^{-1}=(1234567)$, i.e. 
 \begin{equation*} 
 \Frob P_1 =P_2, \quad \Frob P_2 = P_3, \quad \ldots \quad  \Frob P_6=P_7, \quad  \Frob P_7 =P_1
 \end{equation*}
 We may thus obtain such a septuple by choosing a point $P_1 \in \PP^2$ which is defined
 over $\FF_{q^7}$ but not over $\FF_{q}$ and then acting by $\Frob$. Thus to count
 $|X^{\Frob \sigma}|$ we need to count $|\PP^2(\FF_{q^7})|$ and $|\PP^2(\FF_{q})|$
 and take the difference:
 \begin{equation*}
  |X^{\Frob \sigma}| = |\PP^2(\FF_{q^7})| - |\PP^2(\FF_q)| = (q^{14}+q^7+1)-(q^2+q+1)=q^{14}+q^7-q^2-q
 \end{equation*}
 
 In order to determine $|\Delta^{\Frob \sigma}|$ we observe that if $(P_1, \ldots, P_7)$ is a septuple
 in $X^{\Frob \sigma}$ such that three of the points lie on a line $L$, then $L$ must be defined over
 $\FF_q$ and $L$ must contain all seven points $P_1, \ldots, P_7$ (see Lemma 4 of \cite{bergvall_pts} for some
 elaboration on this fact). Thus, to determine $|\Delta^{\Frob \sigma}|$ we should count the number of
 lines $L \subset \PP^2$ defined over $\FF_q$ and the number of ways to pick a point
 $P_1$ on $L$ such that $P_1$ is defined over $\FF_{q^7}$ but not over $\FF_q$. 
 The number of $\FF_q$-lines $L$ in $\PP^2$ is $q^2+q+1$ and the number of ways to pick $P_1$
 on $L$ is $(q^7+1)-(q+1)=q^7-q$.
 We thus obtain the result
 \begin{equation*}
  |\Delta^{\Frob \sigma}| = (q^2+q+1)(q^7-q)=q^9 + q^8 + q^7 - q^3 - q^2 - q
 \end{equation*}
 We get
 \begin{equation*}
  |X^{\Frob \sigma}| -|\Delta^{\Frob \sigma}| = q^{14} - q^9 - q^8 + q^3
 \end{equation*}
 and, finally
 \begin{equation*}
  |\Ccal_{2,7}^{\Frob \sigma}| = 
  \frac{|X^{\Frob \sigma}| -|\Delta^{\Frob \sigma}|}{|\mathrm{PGL}(3,\FF_q)|} =
  \frac{q^{14} - q^9 - q^8 + q^3}{(q^2+q+1)(q^3-q)(q^3-q^2)} =
  q^6 + q^4 + q^3 + q^2 + 1
 \end{equation*}
\end{proof}

\begin{center}
\begin{table}[h]
\begin{tabular}{ l | l  }
$[\sigma]$ & $|\Ccal_{2,7}^{\Frob \sigma}|$ \\
\hline
$7$ & $q^6 + q^4 + q^3 + q^2 + 1$ \\
$61$ & $q^6 + q^5 + q^4 - q^2$ \\
$52$ & $q^6 + q^4 - q^3 - q$ \\
$51^2$ & $q^6 + 2q^5 + 3q^4 + 3q^3 + 2q^2 + q$ \\
$43$ & $q^6 - q^5 - q^4 + q^3$ \\
$421$ & $q^6 - 3q^4 + 2$ \\
$41^3$ & $q^6 + 2q^5 - q^4 - 2q^3$ \\
$3^21$ & $q^6 - q^5 - q^4 - 8q^3 + 9q^2 + 6q + 18$ \\
$32^2$ & $q^6 - q^5 - q^4 + q^3$ \\
$321^2$ & $q^6 - q^5 - q^4 + q^3$ \\
$31^4$ & $q^6 - q^5 - q^4 + q^3$ \\
$2^31$ & $q^6 - 2q^5 - 11q^4 + 18q^3 + 38q^2 - 36q - 48$ \\
$2^21^3$ & $q^6 - 4q^5 - q^4 + 16q^3 - 6q^2 - 12q + 6$ \\
$21^5$ & $q^6 - 10q^5 + 41q^4 - 86q^3 + 90q^2 - 36q$ \\
$1^7$ & $q^6 - 28q^5 + 323q^4 - 1952q^3 + 6462q^2 - 11004q + 7470$
\end{tabular}
\caption{$S_7$-equivariant point counts of the space of $7$ points in general linear position in $\PP^2$.}
\label{counttab}
\end{table}
\end{center}

If we view $\Hetc{k}{\Ccal_{2,7}}$ as a representation of $S_7$, then we may
read off the values of the corresponding character $\chi_{k}$ from Table~\ref{counttab} as
\begin{equation*}
 \chi_k(\sigma) = (-1)^k|\Ccal_{2,7}^{\Frob \sigma}|_{q^{6-k}}
\end{equation*}
where $|\Ccal_{2,7}^{\Frob \sigma}|_{q^{6-k}}$ denotes the coefficient of
$q^{6-k}$ in the polynomial $|\Ccal_{2,7}^{\Frob \sigma}|_{q^{6-k}}$.

It is a standard fact from representation theory of finite groups that the
irreducible representations of $S_m$ are indexed by partitions of $m$.
In Table~\ref{cohtab} we decompose the cohomology groups
$\Hetc{k}{\Ccal_{2,7}}$ into irreducible representations of $S_7$.
A number $r$ in the column indexed by $7^{n_7}6^{n_6} \cdots 2^{n_2}1^{n_1}$ means
that the irreducible representation corresponding to $7^{n_7}6^{n_6} \cdots 2^{n_2}1^{n_1}$ occurs
with multiplicity $r$ in $\Hetc{k}{\Ccal_{2,7}}$. 

\begin{center}
\begin{table}[h]
\begin{tabular}{r|rrrrrrrrrrrrrrr}
\, & {\tiny $7$} & {\tiny $6$} & {\tiny $52$} & {\tiny $51^2$} & {\tiny $43$} & {\tiny $421$} & {\tiny $41^3$} & {\tiny $3^21$} & {\tiny $32^2$} & {\tiny $321^2$} & {\tiny $31^4$} & {\tiny $2^31$} & {\tiny $2^21^3$} & {\tiny $21^5$} & {\tiny $1^7$} \\
\hline
$H^0$ & 1 & 0 & 0 & 0 & 0 & 0 & 0 & 0 & 0 & 0 & 0 & 0 & 0 & 0 & 0 \\
$H^1$ & 0 & 0 & 1 & 0 & 1 & 0 & 0 & 0 & 0 & 0 & 0 & 0 & 0 & 0 & 0 \\
$H^2$ & 0 & 1 & 1 & 3 & 1 & 3 & 1 & 3 & 1 & 1 & 0 & 0 & 0 & 0 & 0 \\
$H^3$ & 0 & 3 & 6 & 9 & 7 & 15 & 10 & 9 & 6 & 12 & 3 & 5 & 3 & 0 & 0 \\
$H^4$ & 3 & 9 & 21 & 19 & 20 & 47 & 27 & 25 & 29 & 42 & 20 & 17 & 13 & 6 & 1 \\
$H^5$ & 3 & 14 & 34 & 31 & 31 & 78 & 42 & 44 & 48 & 75 & 34 & 30 & 29 & 13 & 1 \\
$H^6$ & 2 & 9 & 18 & 25 & 23 & 50 & 31 & 34 & 28 & 52 & 19 & 23 & 22 & 9 & 4
\end{tabular}
\caption{The cohomology groups of $\Ccal_{2,7}$ as representations of $S_7$.}
\label{cohtab}
\end{table}
\end{center}

\section{Eight points}

In this section we will show, using techniques similar to those of the proof of Proposition~\ref{27minpure_prop},
that the space of eight points in general position is minimally pure. 
We will also explain why similar
techniques will not suffice to prove minimal purity for the space of eight points in general linear position
(of course, this does not prove that the space of eight points in general linear position is not minimally pure).

Recall that $8$ points $P_1, \ldots, P_8 \in \PP^2$ are in general position if
\begin{itemize}
 \item[(\emph{i})] no three of the points lie on a line,
 \item[(\emph{ii})] no six of the points lie on a conic, and
 \item[(\emph{iii})] the eight points do not lie on a singular cubic with a singularity at one of the points $P_1, \ldots, P_8$.
\end{itemize}
We denote the space of octuples $(P_1, \ldots, P_8)$ of points in $\PP^2$ in general position, up
to projective equivalence, by $\Ccal_{2,8}^{\mathrm{gp}}$.

\begin{prop}
\label{28minpure_prop}
 The space $\Ccal_{2,8}^{\mathrm{gp}}$ of octuples of points in the plane in general position
 is minimally pure.
\end{prop}

\begin{proof}
 The idea of the proof is the same as in the proof of Proposition~\ref{27minpure_prop}:
 we shall construct a $\FF_q$-rational and $S_8$-equivariant finite cover
 $\varphi: X\to \Ccal_{2,8}^{\mathrm{gp}}$ such that $X$ is minimally pure
 and the result will then follow (in fact, in this case the cover will even be
 equivariant with respect to the much larger group $W(E_8)$, the Weyl group of the root system $E_8$).

 To construct the space $X$ we observe that if we blow up $\PP^2$ in eight
 points $P_1, \ldots, P_8$ in general position we obtain a geometrically marked Del Pezzo
 surface $S$ of degree $1$. Let $\pi:S \to \PP^2$ denote the blow down morphism.
 We rigidify the situation further by marking $S$ with an anticanonical curve $C$,
 i.e. a curve whose class is $3L-E_1-\cdots-E_8$, where $L$ denotes the strict transform of a line and
 $E_i$ denotes the exceptional class corresponding to the point $P_i$. 
 The curve $A=\pi(C)$ has degree $3$ and passes through the points 
 $P_1, \ldots, P_8$. Suppose that $A$ is reducible. Then $A$ contains a component $A_1$ of
 degree $1$. Since $P_1, \ldots, P_8$ are in general position, the component $A_1$ can contain 
 at most $2$ of the points. Thus, the residual curve $A_2$ must contain $6$ of the points.
 But $A_2$ has degree $2$ and is therefore allowed to contain at most $5$ of the points.
 We conclude that $A$ must be irreducible. 
 
 Let $X$ be the space of triples $(S,B,C)$ such that $S$ is a weak Del Pezzo surface,
 $B$ is a geometric marking of $S$ and $C$ is an anticanonical curve on $S$ such
 that $A=\pi(C)$ is singular.
 The morphism $\varphi:X \to \Ccal_{2,8}^{\mathrm{gp}}$ sending a triple
 $(S,B,C)$ to the octuple $(P_1,\ldots, P_8)$ defined by the blow down of $B$
 is clearly $\FF_q$-rational, finite (of degree $12$) and $S_8$-equivariant. It thus remains to
 show that $X$ is minimally pure.
 
 The curve $A$ will either be a nodal cubic or a cuspidal cubic.
 Let
 $X^{\textrm{node}}$ denote the subspace of $X$ consisting of triples
 $(S,B,C)$ such that $A$ is nodal and let
 $X^{\textrm{cusp}}$ be the subspace of $X$ consisting of triples such that
 $A$ is cuspidal. We then have that $X$ is the disjoint union
 $$
  X = X^{\textrm{node}} \sqcup X^{\textrm{cusp}}.
 $$
 
 The next step is to show that $X^{\textrm{node}}$ is isomorphic to a complement of
 a toric arrangement in an algebraic torus $T$ and that $X^{\textrm{cusp}}$ is isomorphic to a 
  projectivization of a complement of a hyperplane arrangement in a vector space $V$
 in such at way that $V$ can be naturally be identified with the tangent space of $T$ at the identity.
 This follows exactly as in Proposition 1.8, Proposition 1.11 and Proposition 1.17 of \cite{looijenga}.
 It now directly follows from Lemma 3.6 of \cite{looijenga} that there is an inclusion
 $$
 \Hetc{k}{X} \hookrightarrow  \Hetc{k}{X^{\textrm{node}}}
 $$
 Since complements of toric arrangements are  minimally pure it follows that also $X$ is minimally pure. This completes the proof.
\end{proof}

\begin{rem}
 The cohomology groups of the space $X^{\mathrm{cusp}}$ (at least before the projectivization)
 were computed as representations of the Weyl group $W(E_8)$ by Fleischmann and Janiszczak in
 \cite{fleischmannjaniszczak}. The cohomology groups of $X^{\mathrm{node}}$ were computed as
 representations of $W(E_8)$ by the author in \cite{bergvalltor}. Combining these two results using
 Lemma 3.6 of \cite{looijenga} gives the cohomology groups of $X$ as representations of $W(E_8)$.
\end{rem}

Since the space of eight points in general position is 
isomorphic to the moduli space of geometrically marked Del Pezzo surfaces of degree $1$
we can reformulate the above result as follows.

\begin{cor}
 The moduli space of geometrically marked Del Pezzo surfaces of degree $1$ is minimally pure.
\end{cor}

Thus, it is possible to determine the cohomology groups of $\Ccal_{2,8}^{\mathrm{gp}}$
by counting points over finite fields. This is however rather nontrivial and the topic
of ongoing research.

One central aspect of the proofs of Propositions~\ref{27minpure_prop} and \ref{28minpure_prop}
has been the ability to find rational cubics through the points $P_1, \ldots, P_m$
in a consistent way. More precisely, we have in the general case been able to
find rational cubics $A$ of fixed singularity type such that $\mathrm{Jac}(A) \cong \mathbb{G}_m$
and in the non-general case we have also had rational cubics $A$ of a single fixed singularity type
but such that $\mathrm{Jac}(A) \cong \mathbb{G}_a$. 

If we now consider $8$ points in general linear position we can no longer achieve this type of consistency.
For a general octuple $(P_1,\ldots, P_8)$ we will only have irreducible cubic curves passing through
the points, i.e. there is no reducible cubic passing through them. However, as soon as $7$ of the points
$(P_1, \ldots, P_8)$ lie on a conic there will be no irreducible cubic passing through the octuple (otherwise the conic and the irreducible cubic would intersect in $\geq 7$ points which contradicts Bézout's theorem).
We can still cover $\Ccal_{2,8}$ with a space $X$ which decomposes as $X=\cup_{s\in S} X^{s}$ with each  $X^ss$ isomorphic to a complement of an arrangement of tori or hyperplanes - but not in a way so that we are able show that the covering map is proper and $X$ is minimally pure. More precisely, we either get a situation where $X$ clearly is minimally pure but the map is not proper or a situation where the covering map is proper but the cohomological properties of the space $X$ are
not so easily deduced. In the latter case, it is possible to compute the cohomology groups of the
spaces $X^s$ so one possible way to go further in the pursuit of the cohomology of $\Ccal_{2,8}$ could be
to try to compute the cohomology of the union using Gysin maps. This is however likely to be rather complicated
and it seems unlikely that the space $X$ will end up minimally pure.
We remark in passing that Glynn's non-equivariant point count of $\Ccal_{2,8}$ is compatible
with $\Ccal_{2,8}$ being minimally pure but that it is still possible that an equivariant point
count could rule out minimal purity.

The above being said, we can relax the conditions in the definition of octuples in general position a bit
and still achieve minimal purity. More precisely we can allow up to $6$ of the points to lie on a conic and
we can also allow one of the points to lie on a singularity of a cubic passing through the eight points
(but we cannot allow both to happen for the same octuple). 
We can also use the techniques of the present paper to prove minimal purity for the space of
eight points in ``almost general position'' in the sense of Dolgachev and Ortland, see \cite{dolgachevortland} p. 67.
This suggests that the answer to Question~\ref{question} is likely to be rather subtle even in the simplest cases.

\subsection*{Acknowledgements}
I would like to thank Jonas Bergström, Frank Gounelas, Nathan Kaplan, Oliver Leigh and Dan Petersen for interesting discussions and
helpful comments.
\bibliographystyle{plain}

\renewcommand{\bibname}{References}

\bibliography{references} 

\begin{thebibliography}{10}

\bibitem{bergstrombergvall}
{Bergstr{\"o}m, J. and Bergvall, O.}
\newblock {The equivariant Euler characteristic of $\mathcal{A}_3[2]$}.
\newblock {\em {Annali della Scuola Normale Superiore di Pisa, Classe di
  Scienze}}, {To appear, arXiv:1804.09628}, {2018}.

\bibitem{bergvalltor}
{Bergvall, O.}
\newblock {Cohomology of Complements of Toric Arrangements Associated to Root
  Systems}.
\newblock arXiv:1601.01857, 2016.

\bibitem{bergvall_gd}
{Bergvall, O.}
\newblock Equivariant cohomology of moduli spaces of genus three curves with
  level two structure.
\newblock {\em Geom. Dedicata}, 202:165--191, 2019.

\bibitem{bergvall_pts}
{Bergvall, O.}
\newblock {Equivariant Cohomology of the Moduli Space of Genus Three Curves
  with Symplectic Level Two Structure via Point Counts}.
\newblock {\em {European Journal of Mathematics}}, 6:262--320, {2020}.

\bibitem{bergvallgounelas}
{Bergvall, O.} and {Gounelas, F.}
\newblock {Cohomology of moduli spaces of Del Pezzo surfaces}.
\newblock arXiv:1904.10249, 2019.

\bibitem{das}
{Das, R.}
\newblock {Arithmetic statistics on cubic surfaces}.
\newblock {arXiv:2002.11183}, {2020}.

\bibitem{dasoconnor}
{Das, R. and O'Connor, B.}
\newblock {Configurations of noncollinear points in the projective plane}.
\newblock {arXiv:1904.11409}, {2019}.

\bibitem{dimcalehrer}
{Dimca, A. and Lehrer, G.}
\newblock {Purity and Equivariant Weight Polynomials}.
\newblock In {Lehrer, G.I.}, editor, {\em {Algebraic Groups and Lie Groups}},
  {Australian Mathematical Society Lecture Series}, pages {161--182}.
  {Cambridge University Press}, {1997}.

\bibitem{dolgachevortland}
{Dolgachev, I.} and {Ortland, D.}
\newblock {Point Sets in Projective Spaces}.
\newblock {\em {Ast{\'e}risque}}, 165:{1--210}, {1988}.

\bibitem{fleischmannjaniszczak}
{Fleischmann, P. and Janiszczak, I.}
\newblock {Combinatorics and {P}oincar{\'e} polynomials of hyperplane
  complements for exceptional {W}eyl groups}.
\newblock {\em J. Combin. Theory Ser. A}, 63(2):257--274, 1993.

\bibitem{glynn}
{Glynn, D.}
\newblock Rings of geometries. {II}.
\newblock {\em J. Combin. Theory Ser. A}, 49(1):26--66, 1988.

\bibitem{grosshackingkeel}
{Gross, M., and Hacking, P., and Keel, S.}
\newblock {Moduli of surfaces with an anti-canonical cycle}.
\newblock {\em {Compos. Math.}}, {151}({2}):{265--291}, {2015}.

\bibitem{iampolskaiaetal}
{Iampolskaia, A. and Skorobogatov, A. and Sorokin, E.}
\newblock Formula for the number of {$[9,3]$} {MDS} codes.
\newblock volume~41, pages 1667--1671. 1995.
\newblock Special issue on algebraic geometry codes.

\bibitem{kaplanetal}
{Kaplan, N. and Kimport, S. and Lawrence, R. and Peilen, L. and Weinreich, M.}
\newblock Counting arcs in projective planes via {G}lynn's algorithm.
\newblock {\em J. Geom.}, 108(3):1013--1029, 2017.

\bibitem{kaplanmatei}
{Kaplan, N. and Matei, V.}
\newblock {Counting Plane Cubic Curves over Finite Fields with a Prescribed
  Number of Rational Intersection Points}.
\newblock arXiv:2003.13944, 2020.

\bibitem{looijenga81}
{Looijenga, E.}
\newblock {Rational surfaces with an anticanonical cycle}.
\newblock {\em {Ann. of Math. (2)}}, {114}({2}):{267--322}, {1981}.

\bibitem{looijenga}
{Looijenga, E.}
\newblock {Cohomology of $\mathcal{M}_3$ and $\mathcal{M}_3^1$}.
\newblock In {B{\"o}digheimer, C.-F.} and {Hain, R.M.}, editors, {\em {Mapping
  Class Groups and Moduli Spaces of Riemann Surfaces}}, volume {150} of {\em
  {Contemporary Mathematics}}, pages {205--228}, {1993}.

\end{thebibliography}
\end{document}